\documentclass[11pt]{amsart}
\usepackage{amsmath, amsfonts, amsthm, amssymb}
\usepackage[all,cmtip,line]{xy}
\usepackage{array}
\usepackage{enumerate}
\usepackage{srcltx} % for inverse search in LINUX -- does not cause any problems in Windows

\newtheorem{thm}{Theorem}[section]
\newtheorem{conj}{Conjecture}[section]
\newtheorem{lem}{Lemma}[section]
\newtheorem{prop}[lem]{Proposition}
\newtheorem{cor}[lem]{Corollary}
\newtheorem{defn}[lem]{Definition}
\newtheorem{rem}[lem]{Remark}

\numberwithin{equation}{section}

\newcommand \eps{\varepsilon}

\newlength{\originalbase}
\newcommand{\spacing}[1]{\setlength{\baselineskip}{#1\originalbase}}
\begin{document}               % PLUS THE \end COMMAND AT THE END.

\newcommand{\avint}{{- \hspace{-3.5mm} \int}}

\spacing{1}

%--------Meta Data: Fill in your info------
\title{Gradient estimate for complex Monge-Ampere equation with continuous right hand side}
\author{Xiuxiong Chen, Jingrui Cheng}

\maketitle
\begin{abstract}
In this note, we consider complex Monge-Ampere equation posed on a compact K\"ahler manifold. We show how to get $L^p$($p<\infty$) and $L^{\infty}$ estimate for the gradient of the solution in terms of the continuity of the right hand side.
\end{abstract}
\section{Introduction}

The famous Calabi's volume conjecture asks: For any compact K\"ahler manifold $(M, [\omega_0])$ and for any smooth closed $(1,1)$ form $\alpha\in c_1(M)$, can we find a K\"ahler form $\omega_{\varphi}:=\omega_0+\sqrt{-1}\partial\bar{\partial}\varphi\in[\omega_0]$, such that $Ric(\omega_{\varphi})
=\alpha$? 

It is not hard to see that this problem can be reduced to a problem of solving Monge-Amp$\grave{\text{e}}$re equations: 
Given a function $F$, can we find a K\"ahler form
$\omega_\varphi = \omega_0 +\sqrt{-1} \partial \bar \partial \varphi  \in [\omega_0] $ such that  
\begin{equation}
\omega_\varphi^n = e^F \omega_0^n
\end{equation}
provide that 
\[
\displaystyle \int_M\;e^F \omega_0^n =  \displaystyle \int_M\; \omega_0^n. 
\]
Yau first solved this equation for any $F\in C^3(M)$, in his celebrated solution of Calabi conjecture\cite{Yau78}.   The study of this equation (1.1) has been very intense over the last few decades with various regularity estimates obtained when the right hand side is less smooth.  Here we list a few results (by no means complete)  to motivate our present study. \\

\begin{enumerate}
\item The deep work of Kolodziej which states that if $e^F \in L^p(M)$ for some $p > 1$, then $\varphi \in C^\alpha$ for some $\alpha \in (0,1).$
\item  $C^{2,\alpha}$ estimate when the right hand side $F\in C^{\alpha}$. For real Monge Amp$\grave{\text{e}}$re equation, this  goes back to Caffarelli\cite{Ca1}. For complex Monge-Amp$\grave{\text{e}}$re equation, following the theory of  Evans-Krylov(see \cite{Wang} for details on extension to complex setting), we know $[\varphi]_{C^{2,\alpha'}(M,g)}$ is uniformly
bounded, for each $\alpha'<\alpha$, if $\Delta\varphi$ is uniformly bounded.

\item The important theorem of Chen-He \cite{Chen-He} which proves that if $F \in W^{1,p} $ for $p > 2n$, then we have $n+\Delta \varphi < C.\;$ 
\end{enumerate}

All the above results either estimate the $C^0$ bound(or continuity bound of $\varphi$), or estimate the complex Hessian of $\varphi$. 
 In this note, we will obtain estimates which, in a sense, are in between of these results. Namely we show that it is possible to obtain estimate for the gradient of $\varphi$, when the right hand side $F$ is only assumed to be have certain continuity conditions. Let $\omega(r)$ be the modulus of continuity of $F$. That is, for $r>0$, define $\omega(r)=\sup_{x,y\in M, d(x,y)<r}|F(x)-F(y)|$. Here $d(x,y)$ is the distance function in terms of the smooth background metric $\omega_0$.
The precise results we will obtain are the following apriori estimates:
\begin{thm}\label{THM}
Let $(M,\omega_0)$ be a compact K\"ahler manifold of complex dimension $n$. Let $\varphi$ be a $\omega_0$-plurisubharmonic function solving $(\omega_0+\sqrt{-1}\partial\bar{\partial}\varphi)^n=e^F\omega_0^n.$ Then we have:
\begin{itemize}
\item (i)Assume $F$ is continuous, then for any $p<\infty$, we have $||\nabla\varphi||_{L^p(\omega_0^n)}\le C.$ Here the constant $C$ depends on the background metric $\omega_0$, absolute bound of $F$, modulus of continuity $\omega(r)$ of $F$, $p$ and $n$.
\item (ii)When $n=2$, the above constant $C_1$ does not depend on the continuity of $F$.
\item (iii)If the modulus of continuity of $F$ satisfies $\int_0^1\frac{\omega^2(r)}{r}dr<\infty$, then we have $|\nabla\varphi|\le C_2$ pointwise. Here the constant $C_2$ depends on the background metric $\omega_0$, the absolute bound of $F$, an upper bound for $\int_0^1\frac{\omega^2(r)}{r}dr$ and $n$. 
\end{itemize}
\end{thm}
As a direct consequence of the point (iii) above, we can conclude that:
\begin{cor}
If $F\in C^{\alpha}$ for some $\alpha>0$, then we have $|\nabla\varphi|\le C_3$ pointwise. Here the constant $C_3$ depends on the background metric $\omega_0$, $C^{\alpha}$ norm of $F$ and $n$.
\end{cor}
Indeed, this follows immediately from point (iii) if one takes $\omega(r)=Ar^{\alpha}$. 

\begin{rem}
It is straighforward to generalize the results here to the case of Dirichlet problem on bounded domains, provided that the domain and boundary data is ``reasonable" which allows one to get global $C^0$ estimate and boundary gradient estimate. 
\end{rem}

There have been some previous works on the gradient estimate for complex Monge-Amp$\grave{\text{e}}$re. In \cite{Blocki}, the author showed how to get $C^1$ estimate directly from $C^0$ estimate, using a maximum principle argument, but this estimate will require a Lipschitz bound on the right hand side. This estimate has been improved in \cite{Chen-He}, in which the assumption is weakened to $\nabla F\in L^p$ with $p>2n$ and the $C^{1,1}$ bound is obtained. 

In this note, we obtain gradient estimates without assuming any bound on the derivatives of the right hand side, instead, we only impose some continuity assumptions. Next we briefly explain the ideas.

Let $\Delta_{\varphi}$ be the Laplace operator defined by the metric $\omega_0+\sqrt{-1}\partial\bar{\partial}\varphi$. Roughly speaking, we will calculate the term $\Delta_{\varphi}(|\nabla\varphi|^2)$.
In the outcome of the calculation, the term involving the right hand side is given by $2\nabla\varphi\cdot\nabla F$. Since we do not wish to assume any bound on $\nabla F$, we will do integral estimates and integrating by parts in this term. When $n=2$, this turns out to be sufficient to close the estimate and we only use the absolute bound of $F$ and not its continuity. As a result, we get the estimate said in point (ii) of Theorem \ref{THM}. For general dimension, we need to decompse $F$ as $F=F_g+F_b$(good plus bad), where $F_g$ has bounded gradient and $F_b$ has small $C^0$ bound, which uses the continuity of $F$. Then we need to integrate by parts in the term $\nabla\varphi\cdot\nabla F_b$, and the smallness of $F_b$ in $C^0$ norm allows us to close the estimate. These will be done in section 2. 
In order to get $L^{\infty}$ estimate, we need to use the Nash-Moser iteration, and we need to carefully keep track of how $|\nabla F_g|$ grows when we take $F_b$ to have smaller and smaller $C^0$ norm. This accounts for the integrability condition for the modulus of continuity required in point (iii) of Theorem \ref{THM}. We will carry out the details in section 3.

This result is still very far from the optimal result one can hope for. Indeed, from the deep work of Caffarelli \cite{Ca1}, \cite{Ca2} on the Dirichlet problem for real Monge-Amp$\grave{\text{e}}$re equations(or when the solution is strictly convex), it looks quite reasonable to believe the same results would hold for complex analogue as well. That is, it may be possible to get $C^{1,\alpha}$ estimate of $\varphi$ when $F$ is only bounded and also $C^{2,\alpha}$ estimate of $\varphi$ when $F\in C^{\alpha}$. This question is already raised in \cite{Chen-He}. There have been some partial results which shows the solution is $C^{2,\alpha}$ by assuming the solution already has some absolute bound for the complex Hessian, see the work of \cite{Chen-Wang}, \cite{DZZ} and \cite{Wang}. The work \cite{CLZ} even shows that it suffices to assume the solution is in $C^{1,\beta}$($\beta$ close to 1) in order for improving the regularity to $C^{2,\alpha}$. However, the question of getting $C^{2,\alpha}$ estimate without any extra regularity assumption on the solution remains open and seems out of reach for the moment.

\begin{conj}  Do we have $C^{1,\alpha}$ estimate of $\varphi$ when $F$ is only bounded and also $C^{2,\alpha}$ estimate of $\varphi$ when $F\in C^{\alpha}?$
\end{conj}

We would like to emphasize that we are working in compact manifold without boundary. On local $\mathbb{C}^n$, one can construct Pogorelov type counterexamples. Indeed, one may consider:
\[
 u(z_1, z_2,\cdots z_n) = n^{2\over n} (1 +|z_1|^2 + \cdots |z_{n-1}|^2) |z_n|^{2\over n}
\]
It is shown in \cite{He12} that this function solves the complex Monge-Amp$\grave{\text{e}}$re equation
\begin{equation}
\det u_{i\bar j} = 1
\end{equation}
in the viscosity sense in $\mathbb{C}^n.\;$. This example shows that there is no interior gradient bound at least for $n\ge 3$. Thus, for our conjecture to be true,  the global nature of compact K\"ahler manifold must come into play in a crucial way. \\

\section{Estimate of the gradient in $L^p$}
The goal of this section is to prove point (i) and (ii) in the Theorem \ref{THM}. Now we set up the notations. In the following $\Delta$ means the Laplace operator defined by the background metric $\omega_0$, namely $\Delta=g^{i\bar{j}}\partial_{i\bar{j}}$ and $\omega_0=\sqrt{-1}g_{i\bar{j}}dz_i\wedge d\bar{z_j}$ in local coordinates. $\Delta_{\varphi}$ means the Laplace operator defined by the metric $\omega_{\varphi}:=\omega_0+\sqrt{-1}\partial\bar{\partial}\varphi$. $|\nabla\cdot|$ means the gradient is taken under the metric $\omega_0$, and $|\nabla_{\varphi}\cdot|^2_{\varphi}$ means the gradient is taken under $\omega_{\varphi}$.

First we will do some calculations which will work in any dimension, then we will see how the dimension 2 assumption will help us. Before we start the calculation, we note that the $C^0$ bound of $\varphi$ can be estimated by the $C^0$ bound of $F$, according to the result by Kolodziej(or one can use Alexandrov maximum principle for an elementary proof.)

Let $H(\varphi)=-C\varphi+\delta\varphi^2$. First we compute
\begin{equation}\label{2.1}
\begin{split}
\Delta_{\varphi}\big(&e^{H(\varphi)}(|\nabla\varphi|^2+K)\big)=\Delta_{\varphi}(e^{H(\varphi)})(|\nabla\varphi|^2+K)+e^{H(\varphi)}\Delta_{\varphi}(|\nabla\varphi|^2)\\
&+2e^{H(\varphi)}H'(\varphi)\nabla_{\varphi}\varphi\cdot_{\varphi}\nabla_{\varphi}(|\nabla\varphi|^2).
\end{split}
\end{equation}
To proceed further, we need a computation and estimate for $\Delta_{\varphi}(|\nabla\varphi|^2)$:
\begin{lem}
There exists a constant $C_0$, depending only on the bisectional curvature of  background metric $\omega_0$, such that under normal coordinates at a point $p\in M$, we have
\begin{equation}\label{2.2}
\Delta_{\varphi}(|\nabla\varphi|^2)(p)\ge -C_0tr_{\omega_{\varphi}}\omega_0|\nabla\varphi|^2+\frac{|\varphi_{i\alpha}|^2}{1+\varphi_{i\bar{i}}}+\frac{\varphi_{i\bar{i}}^2}{1+\varphi_{i\bar{i}}}+2\nabla\varphi\cdot\nabla F|_p.
\end{equation}
\end{lem}
\begin{proof}
We recall that, for any point $p\in M$, we can choose coordinate in a neighborhood of $p$ so that $g_{i\bar{j}}(p)=\delta_{ij}$, $\partial_{z_k}g_{i\bar{j}}(p)=0$, $\varphi_{i\bar{j}}(p)$ is diagnal. Here $g_{i\bar{j}}$ is the metric tensor for the K\"ahler metric $\omega_0$, that is, $\omega_0=\sqrt{-1}g_{i\bar{j}}dz_i\wedge d\bar{z_j}$. 
Under this coordinate, we compute:
\begin{equation*}
\begin{split}
\Delta_{\varphi}&(|\nabla\varphi|^2)=\frac{1}{1+\varphi_{i\bar{i}}}(g^{\alpha\bar{\beta}}\varphi_{\alpha}\varphi_{\bar{\beta}})_{i\bar{i}}=\frac{1}{1+\varphi_{i\bar{i}}}\big((g^{\alpha\bar{\beta}})_{i\bar{i}}\varphi_{\alpha}\varphi_{\bar{\beta}}+g^{\alpha\bar{\beta}}(\varphi_{\alpha}\varphi_{\bar{\beta}})_{i\bar{i}}\\
&+2Re\big((g^{\alpha\bar{\beta}})_i(\varphi_{\alpha}\varphi_{\bar{\beta}})_{\bar{i}}\big).
\end{split}
\end{equation*}
Using our assumption of the chosen coordinate, we see that 
\begin{equation*}
(g^{\alpha\bar{\beta}})_{i\bar{i}}=R_{i\bar{i}\alpha\bar{\beta}},\,\,(g^{\alpha\bar{\beta}})_i=0.
\end{equation*}
Hence, when evaluated at point $p$,
\begin{equation*}
\Delta_{\varphi}(|\nabla\varphi|^2)=\frac{R_{i\bar{i}\alpha\bar{\beta}}\varphi_{\alpha}\varphi_{\bar{\beta}}}{1+\varphi_{i\bar{i}}}+\frac{g^{\alpha\bar{\beta}}(\varphi_{\alpha}\varphi_{\bar{\beta}})_{i\bar{i}}}{1+\varphi_{i\bar{i}}}.
\end{equation*}
Moreover,
\begin{equation*}
\begin{split}
\frac{1}{1+\varphi_{i\bar{i}}}&g^{\alpha\bar{\beta}}(\varphi_{\alpha}\varphi_{\bar{\beta}})_{i\bar{i}}=\frac{1}{1+\varphi_{i\bar{i}}}(\varphi_{\alpha}\varphi_{\bar{\alpha}})_{i\bar{i}}=\frac{|\varphi_{\alpha i}|^2}{1+\varphi_{i\bar{i}}}+\frac{|\varphi_{\alpha\bar{i}}|^2}{1+\varphi_{i\bar{i}}}\\
&+2Re\big(\frac{\varphi_{\alpha i\bar{i}}\varphi_{\bar{\alpha}}}{1+\varphi_{i\bar{i}}}\big).
\end{split}
\end{equation*}
Differentiating the equation in the coordinate $\alpha$ and evaluating at point $p$, we find that 
\begin{equation*}
\frac{\varphi_{i\bar{i}\alpha}}{1+\varphi_{i\bar{i}}}=F_{\alpha}.
\end{equation*}
That is,
\begin{equation*}
2Re\big(\frac{\varphi_{\alpha i\bar{i}}\varphi_{\bar{\alpha}}}{1+\varphi_{i\bar{i}}}\big)=2Re(F_{\alpha}\varphi_{\bar{\alpha}})=2\nabla\varphi\cdot\nabla F.
\end{equation*}
Now we can easily estimate:
\begin{equation*}
\begin{split}
\Delta_{\varphi}(|\nabla\varphi|^2)&=\frac{R_{i\bar{i}\alpha\bar{\beta}}\varphi_{\alpha}\varphi_{\bar{\beta}}}{1+\varphi_{i\bar{i}}}+\frac{|\varphi_{i\alpha}|^2}{1+\varphi_{i\bar{i}}}+\frac{\varphi_{i\bar{i}}^2}{1+\varphi_{i\bar{i}}}+2\nabla\varphi\cdot\nabla F\\
&\ge -C_0tr_{\varphi}g|\nabla\varphi|^2+\frac{|\varphi_{i\alpha}|^2}{1+\varphi_{i\bar{i}}}+\frac{\varphi_{i\bar{i}}^2}{1+\varphi_{i\bar{i}}}+2\nabla\varphi\cdot\nabla F.
\end{split}
\end{equation*}
Here the constant $C_0$ depends only on the curvatre bound of the background metric. 
\end{proof}
In the following, we will use $tr_{\varphi}g$ to denote $tr_{\omega_{\varphi}}\omega_0$.

Now we continue the calculation:
\begin{equation}\label{2.3}
\Delta_{\varphi}(e^{H(\varphi)})=e^{H(\varphi)}H'(\varphi)\Delta_{\varphi}\varphi+e^{H(\varphi)}(H''+(H')^2)|\nabla_{\varphi}\varphi|_{\varphi}^2.
\end{equation}
Also
\begin{equation}\label{2.4}
\begin{split}
2\nabla_{\varphi}&\varphi\cdot\nabla_{\varphi}(|\nabla\varphi|^2)=2Re\bigg(\frac{\varphi_{\alpha}\varphi_{\bar{\alpha}\bar{i}}\varphi_{\bar{i}}}{1+\varphi_{i\bar{i}}}\bigg)+2Re\bigg(\frac{|\varphi_i|^2\varphi_{i\bar{i}}}{1+\varphi_{i\bar{i}}}\bigg)\\
&=2Re\bigg(\frac{\varphi_{\alpha}\varphi_{\bar{\alpha}\bar{i}}\varphi_{\bar{i}}}{1+\varphi_{i\bar{i}}}\bigg)+2|\nabla\varphi|^2-2|\nabla_{\varphi}\varphi|_{\varphi}^2.
\end{split}
\end{equation}
After the completion of square:
\begin{equation*}
(H')^2|\nabla_{\varphi}\varphi|_{\varphi}^2|\nabla\varphi|^2+\frac{|\varphi_{i\alpha}|^2}{1+\varphi_{i\bar{i}}}+2H'Re\bigg(\frac{\varphi_{\alpha}\varphi_{\bar{\alpha}\bar{i}}\varphi_{\bar{i}}}{1+\varphi_{i\bar{i}}}\bigg)=\frac{1}{1+\varphi_{i\bar{i}}}|H'\varphi_i\varphi_{\alpha}+\varphi_{i\alpha}|^2.
\end{equation*}
Dropping this square, and denote $u=e^H(|\nabla\varphi|^2+K)$, we use (\ref{2.1}), (\ref{2.2}), (\ref{2.3}) and (\ref{2.4}) to see that: 
\begin{equation*}
\begin{split}
\Delta_{\varphi}&u\ge e^HH'\Delta_{\varphi}\varphi(|\nabla\varphi|^2+K)+e^H(H')^2K|\nabla_{\varphi}\varphi|_{\varphi}^2+e^HH''|\nabla_{\varphi}\varphi|_{\varphi}^2|\nabla\varphi|^2\\
&-C_0e^Htr_{\varphi}g|\nabla\varphi|^2+e^H\frac{\varphi_{i\bar{i}}^2}{1+\varphi_{i\bar{i}}}+2e^H\nabla\varphi\cdot\nabla F+2e^HH'(|\nabla\varphi|^2-|\nabla_{\varphi}\varphi|_{\varphi}^2).
\end{split}
\end{equation*}
Now note that if we choose $2\delta||\varphi||_0=\frac{C}{2}$, we conclude
\begin{equation*}
H'(\varphi)=-C+2\delta\varphi\le-\frac{C}{2}.
\end{equation*}
Hence
\begin{equation*}
\begin{split}
e^H&H'\Delta_{\varphi}\varphi(|\nabla\varphi|^2+K)-C_0e^Htr_{\varphi}g|\nabla\varphi|^2\ge(\frac{C}{2}-C_0)e^Htr_{\varphi}g|\nabla\varphi|^2\\
&-1.5Cne^H(|\nabla\varphi|^2+K)+\frac{CK}{2}e^Htr_{\varphi}g\\
&\ge \frac{C}{4}tr_{\varphi}gu-1.5Cnu.
\end{split}
\end{equation*}
Also if $K>2n$,
\begin{equation*}
e^H\frac{\varphi_{i\bar{i}}^2}{1+\varphi_{i\bar{i}}}=e^H(n+\Delta\varphi)-2ne^H\ge e^H(n+\Delta\varphi)-u.
\end{equation*}
Finally
\begin{equation*}
2e^HH'|\nabla\varphi|^2\ge-3Ce^H|\nabla\varphi|^2\ge -3Cu.
\end{equation*}
So we obtain 
\begin{equation}
\begin{split}
\Delta_{\varphi}&u\ge \frac{C}{4}tr_{\varphi}g u+2\delta e^H|\nabla_{\varphi}\varphi|_{\varphi}^2|\nabla\varphi|^2+e^H(n+\Delta\varphi)+2e^H\nabla\varphi\cdot\nabla F\\
&-(1+4.5C)u+e^H\frac{C^2K}{4}|\nabla_{\varphi}\varphi|_{\varphi}^2.
\end{split}
\end{equation}
Now let $q\ge 1$, we can compute
\begin{equation}\label{1.131}
\begin{split}
&\int_M(q-1)u^{q-2}|\nabla_{\varphi}u|_{\varphi}^2dvol_{\varphi}=\int_Mu^{q-1}(-\Delta_{\varphi}u)dvol_{\varphi}\le -\int_M\frac{C}{4}u^qtr_{\varphi}gdvol_{\varphi}\\
&-\int_Mu^{q-1}2\delta e^H|\nabla_{\varphi}\varphi|_{\varphi}^2|\nabla\varphi|^2dvol_{\varphi}-\int_Mu^{q-1}e^H(n+\Delta\varphi)dvol_{\varphi}\\
&-\int_M2u^{q-1}e^H\nabla\varphi\cdot\nabla Fdvol_{\varphi}+\int_M(1+4.5C)u^qdvol_{\varphi}\\
&-\int_Mu^{q-1}e^H\frac{C^2K}{4}|\nabla_{\varphi}\varphi|_{\varphi}^2dvol_{\varphi}.
\end{split}
\end{equation}
The point of this whole calculation is to handle the term $\nabla\varphi\cdot\nabla F$ above.  We will do this next.

\subsection{When $n=2$}
In this case, the situation is much better and we have the following estimate:
\begin{prop}
When $n=2$, there exists a constant $\eps_0>0$, $C_{2.1}>0$, depending only on the absolute bound of $F$ and the bisectional curvature bound of background metric $\omega_0$, such that 
\begin{equation}
\int_M\exp\big(\eps_0|\nabla\varphi|^2\big)\omega_0^n\le C_{2.1}.
\end{equation}
\end{prop}
We simply integrate by parts for the term involving $\nabla\varphi\cdot\nabla F$:
\begin{equation}\label{1.132}
\begin{split}
-&\int_M2u^{q-1}e^H\nabla\varphi\cdot\nabla Fdvol_{\varphi}=-\int_M2u^{q-1}e^H\nabla\varphi\cdot\nabla(e^F)dvol_g\\
&=\int_M2(q-1)u^{q-2}e^H\nabla u\cdot\nabla\varphi e^Fdvol_g+\int_M2u^{q-1}e^HH'|\nabla\varphi|^2e^Fdvol_g\\
&+\int_M2u^{q-1}e^H\Delta\varphi e^Fdvol_g.
\end{split}
\end{equation}
We handle the terms on the right hand side above one by one. 
\begin{equation*}
\begin{split}
&\int_M2(q-1)u^{q-2}e^H\nabla u\cdot\nabla\varphi e^Fdvol_g\le \int_M2(q-1)u^{q-\frac{3}{2}}e^{\frac{1}{2}H}|\nabla u|dvol_{\varphi}\\
&\le \int_M2(q-1)u^{q-\frac{3}{2}}e^{\frac{1}{2}H}|\nabla_{\varphi}u|_{\varphi}(n+\Delta\varphi)^{\frac{1}{2}}dvol_{\varphi}\le \int_M\frac{q-1}{2}u^{q-2}|\nabla_{\varphi}u|_{\varphi}^2dvol_{\varphi}\\
&+\int_M2(q-1)u^{q-1}e^H(n+\Delta\varphi)dvol_{\varphi}
\end{split}
\end{equation*}
In the first inequality above, we used $e^{\frac{1}{2}H}|\nabla\varphi|\le u^{\frac{1}{2}}$, just from the definition of $u$.

For the second term of (\ref{1.132}), we have
\begin{equation*}
\int_M2u^{q-1}e^HH'|\nabla\varphi|^2e^Fdvol_g\le \int_M2u^{q}\frac{3C}{2}dvol_{\varphi}.
\end{equation*}
This is because $|H'|\le\frac{3C}{2}$, and $e^H|\nabla\varphi|^2\le u$.

Finally, for the last term in (\ref{1.132}), we find,
\begin{equation*}
\int_M2u^{q-1}e^H\Delta\varphi e^Fdol_g\le \int_M2u^{q-1}e^H(n+\Delta\varphi)dvol_{\varphi}.
\end{equation*}
In conclusion, we have
\begin{equation}
\begin{split}
-\int_M&2u^{q-1}e^H\nabla\varphi\cdot\nabla Fdvol_{\varphi}\le \int_M\frac{q-1}{2}u^{q-2}|\nabla_{\varphi}u|_{\varphi}^2dvol_{\varphi}\\
&+\int_M2qu^{q-1}e^H(n+\Delta\varphi)dvol_{\varphi}+\int_M3Cu^qdvol_{\varphi}
\end{split}
\end{equation}
Plug this back to (\ref{1.131}), we find 
\begin{equation}
\begin{split}
&\int_M\frac{q-1}{2}u^{q-2}|\nabla_{\varphi}u|_{\varphi}^2dvol_{\varphi}\le -\int_M\frac{C}{4}tr_{\varphi}gu^qdvol_{\varphi}-\int_M2\delta u^{q-1}e^H|\nabla_{\varphi}\varphi|_{\varphi}^2|\nabla\varphi|^2dvol_{\varphi}\\
&+\int_M(2q-2)u^{q-1}e^H(n+\Delta\varphi)dvol_{\varphi}+\int_M(1+7.5C)u^qdvol_{\varphi}
\end{split}
\end{equation}
Now it is time to use the assumption $n=2$. Then we can write $tr_{\varphi}g=e^{-F}(n+\Delta\varphi)$. Then we can estimate
\begin{equation*}
\begin{split}
&\int_M(2q-2)u^{q-1}e^H(n+\Delta\varphi)dvol_{\varphi}=\int_M(2q-2)e^{qH}(|\nabla\varphi|^2+K)^{q-1}(n+\Delta\varphi)dvol_{\varphi}\\
&\le\int_Me^{qH}(|\nabla\varphi|^2+K)^q(n+\Delta\varphi)dvol_{\varphi}+\int_M(2q-2)^qe^{qH}(n+\Delta\varphi)dvol_{\varphi}.
\end{split}
\end{equation*}
On the other hand, 
\begin{equation*}
\int_M\frac{C}{4}tr_{\varphi}gu^qdvol_{\varphi}=\int_M\frac{C}{4}e^{-F}e^{qH}(|\nabla\varphi|^2+K)^q(n+\Delta\varphi)dvol_{\varphi}.
\end{equation*}
Therefore, if we choose $C$ sufficiently large so that $\frac{C}{4}e^{-F}>2$, we then have 
\begin{equation*}
\begin{split}
&\int_M\frac{q-1}{2}u^{q-2}|\nabla_{\varphi}u|_{\varphi}^2dvol_{\varphi}\le -\int_Me^{qH}(n+\Delta\varphi)(|\nabla\varphi|^2+K)^qdvol_{\varphi}\\
&-\int_M2\delta u^{q-1}e^H|\nabla_{\varphi}\varphi|_{\varphi}^2|\nabla\varphi|^2dvol_{\varphi}+\int_M(1+7.5C)u^qdvol_{\varphi}\\
&+\int_M(2q)^qe^{qH}(n+\Delta\varphi)dvol_{\varphi}.
\end{split}
\end{equation*}
Now we have
\begin{equation*}
\begin{split}
-\int_M&e^{qH}\frac{1}{2}(n+\Delta\varphi)(|\nabla\varphi|^2+K)^qdvol_{\varphi}-\int_M2\delta e^{qH}|\nabla\varphi|^{2q}|\nabla_{\varphi}\varphi|_{\varphi}^2dvol_{\varphi}\\
&\le -\int_M\delta e^{qH}|\nabla\varphi|^{2q+1}dvol_{\varphi}-\int_Me^{qH}e^{\frac{F}{n}}(|\nabla\varphi|^2+K)^qdvol_{\varphi}\\
&\le -\int_Mc_1e^{qH}(|\nabla\varphi|^2+K)^{q+\frac{1}{2}}dvol_{\varphi}.
\end{split}
\end{equation*}
Here $c_1$ is a constant depending only on the $C^0$ bound of $F$ and $\varphi$, but does not depend on $q$.
Hence
\begin{equation*}
\begin{split}
\int_M&\frac{q-1}{2}u^{q-2}|\nabla_{\varphi}u|_{\varphi}^2dvol_{\varphi}+\int_Mc_1e^{qH}(|\nabla\varphi|^2+K)^{q+\frac{1}{2}}dvol_{\varphi}\\
&+\int_M\frac{e^{qH}}{2}(n+\Delta\varphi)(|\nabla\varphi|^2+K)^qdvol_{\varphi}\le \int_M(1+7.5C)u^qdvol_{\varphi}\\
&+\int_M(2q)^{q}e^{qH}(n+\Delta\varphi)dvol_{\varphi}.
\end{split}
\end{equation*}
For the last term above, we have 
\begin{equation*}
\int_M(2q)^qe^{qH}(n+\Delta\varphi)dvol_{\varphi}\le e^{q||H||_0}e^{||F||_0}(2q)^qVol(M).
\end{equation*}
Also we can estimate
\begin{equation*}
\begin{split}
\int_M&(1+7.5C)u^qdvol_{\varphi}=\int_M(1+7.5C)e^{qH}(|\nabla\varphi|^2+K)^qdvol_{\varphi}\\
&\le \int_M\frac{c_1}{2}e^{qH}(|\nabla\varphi|^2+K)^{q+\frac{1}{2}}dvol_{\varphi}+(1+7.5C)^{2q+1}\big(\frac{2}{c_1}\big)^{2q}Vol(M).
\end{split}
\end{equation*}
Hence
\begin{equation*}
\begin{split}
\int_M&\frac{c_1}{2}e^{qH}(|\nabla\varphi|^2+K)^{q+\frac{1}{2}}dvol_{\varphi}\le e^{q||H||_0}e^{||F||_0}(2q)^qVol(M)\\
&+(1+7.5C)^{2q+1}\big(\frac{2}{c_1}\big)^{2q}Vol(M).
\end{split}
\end{equation*}
Now if we put $k=q+\frac{1}{2}$, with $k$ an integer $\ge 2$, we then find that for some $L>1$, $c_2>0$ sufficiently large, we have
\begin{equation*}
\int_M(|\nabla\varphi|^2+K)^kdvol_{\varphi}\le c_2L^kk^k.
\end{equation*}
Now recall the Sterling formula $k^k\approx k!e^k\sqrt{2\pi k}$, we see that we have an integral bound for $\int_M\exp(\eps (|\nabla\varphi|^2+K))dvol_{\varphi}$ for some small $\eps>0$.

\subsection{General dimension}
If the dimension is higher than 2, then we can no longer control the term involving $\Delta\varphi$ in terms of $tr_{\varphi}g$ and we will need to use the continuity of $F$.
The idea is to decompose $F$ into two parts, namely $F=F_g+F_b$(``good" plus ``bad"). Here $\nabla F_g$ is bounded and $F_b$ has small $C^0$ bound. For later use, we need a lemma which shows quantitatively how the gradient bound for $F_g$ and smallness of $F_b$ relate to the continuity of $F$.

\begin{defn}
Let $\omega:[0,1)\rightarrow\mathbb{R}^+$ be a continuous and strictly increasing function with $\omega(0)=0$, such that $|F(x)-F(y)|\le \omega(d_{\omega_0}(x,y))$, for any $x,\,y\in M$ and $d_{\omega_0}(x,y)\le 1$, then we say $\omega(r)$ is a modulus of continuity for the function $F$. Here $d_{\omega_0}(x,y)$ is the distance function defined by the metric $\omega_0$.
\end{defn}
Now we formulate the approximation lemma in terms of the modulus of continuity defined above.
\begin{lem}
Let $F$ be a continuous function on $M$ and let $\omega(r)$ be a modulus of continuity for $F$. Then there exists a constant $C_{2.2}>0$, depending only on the background metric $\omega_0$ , the manifold $M$, and absolute bound for $F$, such that for any $0<r<1$, there exists a function $F_r$ which satisfies the estimates: 
\begin{equation*}
|\nabla F_r|\le\frac{C_{2.2}}{r},\,\,|F-F_r|\le C_{2.2}(\omega(r)+r).
\end{equation*}
\end{lem}
\begin{proof}
The proof of this lemma should be quite standard, but we were unable to find an exact reference. First we can cover the manifold by finitely many coordinate balls $\{B(p_i,2)\}_{i=1}^N$, such that $\cup_{i=1}^NB(p_i,1)$ covers $M$. Let $\{\zeta_i\}_{i=1}^N$ be a partition of unity for the open cover $B(p_i,1)$, namely $\zeta_i$ has compact support in $B(p_i,1)$, nonnegative, and $\sum_{i=1}^N\zeta_i=1$. For each $i$, $F\zeta_i$ has compact support in $B(p_i,1)$ and for $0<r<1$ we define:
\begin{equation*}
F_{i,r}=(F\zeta_i)*\rho_r,\,\,\,F_r=\sum_{i=1}^NF_{i,r}.
\end{equation*}
Here $\rho_r=\rho(\frac{\cdot}{r})$ and $\rho$ is the standard smoothing kernel supported in the unit ball of $\mathbb{C}^n$. Note that each $F_{i,r}$ has compact support in $B(p_i,2)$, hence can be extended smoothly by 0 to $M$. From the standard properties of convolution, we know that for each $i$,
\begin{equation*}
|\nabla F_{i,r}|\le ||F\zeta_i||_0\frac{1}{r}\le\frac{||F||_0}{r},\,\,|F_{i,r}-F\zeta_i|\le \omega_{F\zeta_i}(r).
\end{equation*}
Here $\omega_{F\zeta_i}(r)$ is a modulus of continuity for the function $F\zeta_i$ and it is clear that 
\begin{equation*}
\omega_{F\zeta_i}(r)\le \omega(r)+C_{2.15}||F||_0r.
\end{equation*}
Here the constant $C_{2.15}$ is related to the gradient bound for the cut-off function $\zeta_i$.
Therefore, we have that:
\begin{equation*}
\begin{split}
&|\nabla F|\le \sum_{i=1}^N|\nabla F_{i,r}|\le \frac{N||F||_0}{r},\\
&|F-F_r|\le \sum_{i=1}^N|F\zeta_i-F_{i,r}|\le \sum_{i=1}^N\omega_{F\zeta_i}(r)\le N\omega(r)+NC_{2.15}||F||_0r.
\end{split}
\end{equation*}
\end{proof}
Now we go back to (\ref{1.131}), and we need to use the continuity of $F$ to handle the term involving $\nabla F$. Let $0<r<1$, and $F_r$ be the function given by the above lemma, then
\begin{equation}\label{1.135}
\begin{split}
&\int_M-2u^{q-1}e^H\nabla\varphi\cdot\nabla Fdvol_{\varphi}=-\int_M2u^{q-1}e^H\nabla\varphi\cdot \nabla (e^F-e^{F_r})dvol_g\\
&-\int_M2u^{q-1}e^H\nabla\varphi\cdot\nabla (e^{F_r})dvol_g.
\end{split}
\end{equation}
The handling of the second term is easy. Indeed, choosing $0<r<1$, we then have
\begin{equation}\label{1.136n}\begin{split}
-&\int_M2u^{q-1}e^H\nabla\varphi\cdot\nabla(e^{F_r})dvol_g=-\int_M2u^{q-1}e^H\nabla\varphi\cdot\nabla F_re^{F_r-F}dvol_{\varphi}\\
&\le \int_M2u^{q-1}e^H|\nabla\varphi|\frac{c_3}{r}e^{\omega(r)}dvol_{\varphi}\le \frac{2c_3e}{r}\int_Me^{\frac{1}{2}H}u^{q-\frac{1}{2}}dvol_{\varphi}.
\end{split}
\end{equation}
Here and in the following, the constants $c_i$ will have the dependence as described in point (ii) of the Theorem \ref{THM} and not on $q$. For the first term in (\ref{1.135}), we integrate by parts in the same way as before.
\begin{equation}\label{1.136}
\begin{split}
-&\int_M2u^{q-1}e^H\nabla\varphi\cdot\nabla(e^F-e^{F_r})dvol_g=\int_M2u^{q-1}e^H\Delta\varphi(e^F-e^{F_r})dvol_g\\
&+\int_M2(q-1)u^{q-2}e^H\nabla u\cdot\nabla\varphi(e^F-e^{F_r})dvol_g+\int_M2u^{q-1}e^HH'|\nabla\varphi|^2(e^F-e^{F_r})dvol_g.
\end{split}
\end{equation}
There is no loss of generality to assume that $\omega(r)\ge r$. Hence for the first term on the right hand side of (\ref{1.136}),
\begin{equation}\label{2.14}
\begin{split}
-&\int_M2u^{q-1}e^H\Delta\varphi(e^F-e^{F_r})dvol_g=-\int_M2u^{q-1}e^H\Delta\varphi(1-e^{F_r-F})dvol_{\varphi}\\
&\le \int_M2u^{q-1}e^Hc_4\omega(r)(n+\Delta\varphi)dvol_{\varphi}+\int_M2nc_5\omega(r)u^{q-1}e^Hdvol_{\varphi}
\end{split}
\end{equation}

Now for the second term of (\ref{1.136}), we have
\begin{equation}\label{2.15}
\begin{split}
&\int_M2(q-1)u^{q-2}e^H\nabla u\cdot\nabla\varphi(e^F-e^{F_r})dvol_g=\int_M2(q-1)u^{q-2}e^H\nabla u\cdot\nabla\varphi(1-e^{F_r-F})dvol_{\varphi}\\
&\le \int_M2(q-1)u^{q-2}|\nabla u|e^H|\nabla\varphi|c_4\omega(r)dvol_{\varphi}\le \int_M2(q-1)u^{q-\frac{3}{2}}|\nabla u|e^{\frac{1}{2}H}c_4\omega(r)dvol_{\varphi}\\
&\le \int_M2(q-1)u^{q-\frac{3}{2}}|\nabla_{\varphi}u|_{\varphi}(n+\Delta\varphi)^{\frac{1}{2}}e^{\frac{1}{2}H}c_4\omega(r)dvol_{\varphi}\le \int_M\frac{q-1}{2}u^{q-2}|\nabla_{\varphi}u|_{\varphi}^2dvol_{\varphi}\\
&+\int_M2(q-1)u^{q-1}e^H(n+\Delta\varphi)c_4^2\omega^2(r)dvol_{\varphi}.
\end{split}
\end{equation}
For the last term in (\ref{1.136}), we see that
\begin{equation}\label{2.16}
\int_M2u^{q-1}e^HH'|\nabla\varphi|^2(e^F-e^{F_r})dvol_g\le \int_M2u^{q}\frac{3C}{2}c_5\omega(r)dvol_{\varphi}.
\end{equation}
Combining above calculations (\ref{2.14})-(\ref{2.16}), we obtain from (\ref{1.136})
\begin{equation}
\begin{split}
-&\int_M2u^{q-1}e^H\nabla\varphi\cdot\nabla(e^F-e^{F_r})dvol_g\le \int_M\frac{q-1}{2}u^{q-2}|\nabla_{\varphi}u|_{\varphi}^2dvol_{\varphi}\\
&+ \int_Mu^{q-1}e^H(n+\Delta\varphi)(2c_4\omega(r)+2(q-1)c_4^2\omega^2(r))dvol_{\varphi}\\
&+\int_M3C\omega(r)u^qdvol_{\varphi}+\int_M2nc_5\omega(r)u^{q-1}e^Hdvol_{\varphi}.
\end{split}
\end{equation}
Combining this with (\ref{1.135}) and (\ref{1.136n}), we obtain that for any $0<r<1$, 
\begin{equation}\label{1.139}
\begin{split}
&-\int_M2u^{q-1}e^H\nabla\varphi\cdot\nabla Fdvol_{\varphi}\le \int_M\frac{q-1}{2}u^{q-2}|\nabla_{\varphi}u|_{\varphi}^2dvol_{\varphi}\\
&+\int_Mu^{q-1}e^H(n+\Delta\varphi)\big(2c_4\omega(r)+2(q-1)c_4^2\omega^2(r)\big)dvol_{\varphi}+\int_M3C\omega(r)u^qdvol_{\varphi}\\
&+\int_M2nc_5\omega(r)u^{q-1}e^Hdvol_{\varphi}+\frac{2c_3e}{r}\int_Me^{\frac{1}{2}H}u^{q-\frac{1}{2}}dvol_{\varphi}.
\end{split}
\end{equation}
Plug (\ref{1.139}) back to (\ref{1.131}), we obtain that 
\begin{equation}\label{1.140}
\begin{split}
&\int_M\frac{q-1}{2}u^{q-2}|\nabla_{\varphi}u|_{\varphi}^2dvol_{\varphi}+\int_M\frac{C}{4}u^qtr_{\varphi}g dvol_{\varphi}+\int_Mu^{q-1}2\delta e^H|\nabla_{\varphi}\varphi|_{\varphi}^2|\nabla\varphi|^2dvol_{\varphi}\\
&+\int_Mu^{q-1}e^H(n+\Delta\varphi)\big(1-2c_4\omega(r)-2(q-1)c_4^2\omega^2(r)\big)dvol_{\varphi}\\
&+\int_Mu^{q-1}e^H\frac{C^2K}{4}|\nabla_{\varphi}\varphi|_{\varphi}^2dvol_{\varphi}\le\int_M(1+7.5C)u^qdvol_{\varphi}+\int_M2nc_5\omega(r)u^{q-1}e^Hdvol_{\varphi}\\
&+\frac{2c_3e}{r}\int_Me^{\frac{1}{2}H}u^{q-\frac{1}{2}}dvol_{\varphi}.
\end{split}
\end{equation}
Now we observe that in the above,
\begin{equation}\label{2.20}
\begin{split}
\int_M&u^{q-1}2\delta e^H|\nabla_{\varphi}\varphi|_{\varphi}^2|\nabla\varphi|^2dvol_{\varphi}+\int_Mu^{q-1}e^H\frac{C^2K}{4}|\nabla_{\varphi}\varphi|_{\varphi}^2dvol_{\varphi}\\
&\ge \int_Mc_6u^q|\nabla_{\varphi}\varphi
|_{\varphi}^2dvol_{\varphi}.
\end{split}
\end{equation}
Also
\begin{equation*}
tr_{\varphi}g+|\nabla_{\varphi}\varphi|_{\varphi}^2\ge e^{-\frac{F}{n-1}}(n+\Delta\varphi)^{\frac{1}{n-1}}+|\nabla_{\varphi}\varphi|_{\varphi}^2\ge \frac{1}{2}e^{-\frac{F}{n}}(|\nabla\varphi|^{\frac{2}{n}}+1).
\end{equation*}
Combining the right hand side of (\ref{2.20}) with the term in (\ref{1.140}) involving $tr_{\varphi}g$, we can conclude that 
\begin{equation}\label{1.141}
\begin{split}
\int_M&\frac{q-1}{2}u^{q-2}|\nabla_{\varphi}u|_{\varphi}^2dvol_{\varphi}+\int_Mc_7u^{q+\frac{1}{n}}dvol_{\varphi}+\int_Mu^{q-1}e^H(n+\Delta\varphi)A_q(r)dvol_{\varphi}\\
&\le \int_M(1+7.5C)u^qdvol_{\varphi}+\int_M2nc_5\omega(r)u^{q-1}e^Hdvol_{\varphi}+\frac{2c_3e}{r}\int_Me^{\frac{1}{2}H}u^{q-\frac{1}{2}}dvol_{\varphi}.
\end{split}
\end{equation}
In the above
$$
A_q(r)=1-2c_4\omega(r)-
2(q-1)c_4^2\omega^2(r).
$$
For any fixed $q>1$, we can choose $r$ sufficiently small, so that $A_q(r)>0$. Then we conclude from (\ref{1.141}) that $|\nabla\varphi|^2\in L^{q+\frac{1}{n}}$.

\section{Estimate of the gradient in $L^{\infty}$}
In order to get a bound for $\nabla\varphi$ in $L^{\infty}$, we need to estimate more carefully how the right hand side 
of (\ref{1.141}) grows as $q$ increases.
By doing approximation, we can assume without loss of generality that $\omega:\mathbb{R}_+\rightarrow\mathbb{R}_+$ satisfies $\omega(0)=0$, $\omega(r)\in C^1((0,1])$ and $\omega'(r)>0$.
Define for any $q\ge9$,
\begin{equation}\label{1.142n}
r_q=\sup\{0<r\le 1:c_4\omega(r)\le \frac{1}{\sqrt{8(q-1)}}\}.
\end{equation}
It is easy to verify that for any $r\le r_q$, we have $2c_4\omega(r)\le\frac{1}{4}$, $2(q-1)c_4^2\omega^2(r)\le\frac{1}{4}$. Hence $A_q(r)\ge\frac{1}{2}$.
Choosing $r=r_q$ in (\ref{1.141}), we find that
\begin{equation}\label{1.142}
\begin{split}
&\int_M\frac{q-1}{2}u^{q-2}|\nabla_{\varphi}u|_{\varphi}^2dvol_{\varphi}+\int_Mc_7u^{q+\frac{1}{n}}dvol_{\varphi}+\int_M\frac{1}{2}u^{q-1}e^H(n+\Delta\varphi)dvol_{\varphi}\\
&\le \int_M(1+7.5C)u^qdvol_{\varphi}+\int_M\frac{2nc_5u^{q-1}}{c_4\sqrt{8(q-1)}}e^Hdvol_{\varphi}+\frac{2c_3e}{r_q}\int_Me^{\frac{1}{2}H}u^{q-\frac{1}{2}}dvol_{\varphi}\\
&\le \int_M(1+7.5C)u^qdvol_{\varphi}+\int_M\frac{c_8}{r_q}e^{\frac{1}{2}H}u^{q-\frac{1}{2}}dvol_{\varphi}.
\end{split}
\end{equation}
Again $c_8$ depends on $C^0$ bound of $F$ and $\varphi$ but not on $q$.

On the other hand, if we look at the left hand side of (\ref{1.142}), we see
\begin{equation}\label{3.3n}
\begin{split}
\int_M&\frac{q-1}{2}u^{q-2}|\nabla_{\varphi}u|_{\varphi}^2dvol_{\varphi}+\int_M\frac{1}{2}u^{q-1}e^H(n+\Delta\varphi)dvol_{\varphi}\\
&\ge\int_M\sqrt{q-1}u^{q-\frac{3}{2}}e^{\frac{1}{2}H}|\nabla u|dvol_{\varphi}\ge c_8\int_M\sqrt{q-1}u^{q-\frac{3}{2}}|\nabla u|dvol_g.
\end{split}
\end{equation}
Also we have
\begin{equation}\label{3.4n}
\int_M(1+7.5C)u^qdvol_{\varphi}\le \int_M\frac{c_7}{2}u^{q+\frac{1}{n}}dvol_{\varphi}+\int_Mc_9u^{q-\frac{1}{2}}dvol_{\varphi}.
\end{equation}
Therefore
\begin{equation}\label{3.5n}
c_8\int_M\sqrt{q-1}u^{q-\frac{3}{2}}|\nabla u|dvol_g\le \int_M\big(c_9+\frac{c_8}{r_q}\big)dvol_{\varphi}.
\end{equation}
Hence, using $r_q\le 1$, we conclude by combining (\ref{1.142})-(\ref{3.5n}):
\begin{equation}\label{1.143}
\int_M\sqrt{q-1}u^{q-\frac{3}{2}}|\nabla u|dvol_g\le \int_M\frac{c_{10}}{r_q}u^{q-\frac{1}{2}}dvol_g.
\end{equation}
The constants $c_7$, $c_8$, $c_9$ and $c_{10}$ above depnd only on the $C^0$ bound of $F$ and $\varphi$, but not on $q$.
Putting $\chi=\frac{2n}{2n-1}$, we observe that by Sobolev imbedding, we have
\begin{equation*}
\bigg(\int_Mu^{(q-\frac{1}{2})\chi}dvol_g\bigg)^{\frac{1}{\chi}}\le c_{11}\bigg(\int_M|\nabla(u^{q-\frac{1}{2}})|dvol_g+\int_Mu^{q-\frac{1}{2}}dvol_g\bigg).
\end{equation*}
Hence we conclude from (\ref{1.143}) that for $q\ge 9$,
\begin{equation*}
\begin{split}
\bigg(&\int_Mu^{(q-\frac{1}{2})\chi}dvol_g\bigg)^{\frac{1}{\chi}}\le \big(\frac{c_{11}c_{10}\sqrt{q-\frac{1}{2}}}{r_q}+c_{11}\big)\int_Mu^{q-\frac{1}{2}}dvol_g\\
&\le \frac{c_{12}\sqrt{q}}{r_q}\int_Mu^{q-\frac{1}{2}}dvol_g.
\end{split}
\end{equation*}
Put $q-\frac{1}{2}=\chi^k$, we see that for $k\ge k_0$ (here $k_0$ is the smallest integer for which $\chi^{k_0}-\frac{1}{2}\ge 9$),
\begin{equation*}
||u||_{L^{\chi^{k+1}}}\le c_{12}^{\frac{1}{\chi^k}}(2\chi^k)^{\frac{1}{\chi^k}}(r_{\chi^k+\frac{1}{2}})^{\frac{1}{\chi^k}}||u||_{L^{\chi^k}}.
\end{equation*}
In order to show $L^{\infty}$ bound of $u$, we just need
\begin{equation*}
\sum_{k=k_0}^{\infty}\log\big(c_{12}^{\frac{1}{\chi^k}}(2\chi^k)^{\frac{1}{\chi^k}}(r^{-1}_{\chi^k+\frac{1}{2}})^{\frac{1}{\chi^k}}\big)<\infty.
\end{equation*}
Above is easily seen to be equivalent to showing:
\begin{equation}\label{1.145}
\sum_{k=k_0}^{\infty}\frac{1}{\chi^k}\log(r_{\chi^k+\frac{1}{2}}^{-1})<\infty.
\end{equation}
Here $r_q$ is defined by (\ref{1.142n}). We show that this is guaranteed by the condition: $\int_0^1\frac{\omega(r)^2}{r}dr<\infty$. This follows from the following elementary lemma:
\begin{lem}
Let $\omega(r):[0,1]\rightarrow \mathbb{R}^+$ be a continuous and strictly increasing function such that $\omega(0)=0$. Define $r_q$ as given by (\ref{1.142n}). Let $\chi>1$, then there exists $C_{3.3}$, which depends only on $\chi$, such that
\begin{equation*}
\sum_{k=k_0}^{\infty}\frac{1}{\chi^k}\log(r_{\chi^k+\frac{1}{2}}^{-1})\le C_{3.3}\int_0^1\frac{\omega(r)^2}{r}dr.
\end{equation*}
Here $k_0$ is the minimal integer for which $\chi^{k_0}>9+\frac{1}{2}$.
\end{lem}
\begin{proof}
First we observe that $\int_0^1\frac{\omega^2(r)}{r}dr<\infty$ implies that \begin{equation}\label{1.146}
\int_0^{\omega(1)}s\log\big(\frac{1}{\omega^{-1}(s)}\big)ds<\infty.
\end{equation}
Here $\omega^{-1}(s)$ is the inverse function of $\omega(r)$ on the interval $r\in[0,1]$. 

To see (\ref{1.146}) holds, we can calculate:
\begin{equation*}
\begin{split}
\int_0^{\omega(1)}&s\log\big(\frac{1}{\omega^{-1}(s)}\big)ds=\int_0^{1}\log(r^{-1})\omega(r)\omega'(r)dr=\frac{\omega(r)^2}{2}\log(r^{-1})|_{r=1}-\frac{\omega(r)^2}{2}\log(r^{-1})|_{r=0^+}\\
&+\int_0^{1}\frac{\omega(r)^2}{2r}dr\le \int_0^1\frac{\omega(r)^2}{2r}dr.
\end{split}
\end{equation*}
It only remains to show that (\ref{1.146}) implies (\ref{1.145}). First we observe that from the definition of $r_q$, we have $r_q=\omega^{-1}(c_4^{-1}(8(q-1))^{-\frac{1}{2}})$, hence
\begin{equation}\label{1.147}
r_{\chi^k+\frac{1}{2}}=\omega^{-1}\big(\frac{1}{c_4\sqrt{8(\chi^k-\frac{1}{2})}}\big)\ge \omega^{-1}\big(\frac{1}{4c_4\chi^{\frac{k}{2}}}\big).
\end{equation}
There is no loss of generality to assume $\omega(r)\ge r$(by changing $\omega(r)$ to $\omega(r)+r$), then we have $\omega(1)\ge 1$ and
\begin{equation*}
\begin{split}
\int_0^1&s\log\big(\frac{1}{\omega^{-1}(s)}\big)ds\ge \sum_{k=k_0+1}^{\infty}\int_{\frac{1}{4c_4}\chi^{-\frac{k+1}{2}}}^{\frac{1}{4c_4}\chi^{\frac{-k}{2}}}s\log\big(\frac{1}{\omega^{-1}(s)}\big)ds\\
&\ge \sum_{k=k_0}^{\infty}\big(\frac{1}{4c_4}\chi^{\frac{-k}{2}}-\frac{1}{4c_4}\chi^{-\frac{k+1}{2}}\big)(4c_4\chi^{-\frac{k+1}{2}})\log\big(\frac{1}{\omega^{-1}(\frac{1}{4c_4}\chi^{-\frac{k}{2}})}\big)\\
&\ge\chi^{-\frac{1}{2}}(1-\chi^{-\frac{1}{2}})\sum_{k=k_0}^{\infty}\chi^{-k}\log(r_{\chi^k+\frac{1}{2}}^{-1}).
\end{split}
\end{equation*}
The last inequality used (\ref{1.147}).
\end{proof}

\section{Integrability estimate for $e^{-\varphi}$}
In this section, we will show the following result:
\begin{prop}\label{p4.1}
Let $\varphi$ solves $(\omega_0+\sqrt{-1}\partial\bar{\partial}
\varphi)^n
=e^F\omega_0^n$, with normalization $\sup_M\varphi=0$. Suppose that there exists a function $\Phi:[0,\infty)\rightarrow\mathbb{R}_+$, with $\Phi(t)\rightarrow+\infty$ as $t\rightarrow+\infty$ and $e^F\Phi(e^F)\in L^1(\omega_0^n)$, then for any $p<\infty$, we have $\int_Me^{-p\varphi}\omega_0^n\le C$. Here the constant $C$ depends on $p$, the background metric $\omega_0$, the function $\Phi$ and the integral bound $\int_Me^F\Phi(e^F)\omega_0^n$.
\end{prop}
This result actually follows from the following theorem of Guedj and Zeriahi in \cite{GZ}:
\begin{thm}
For any positive measure $\mu$ which does not charge pluripolar sets, there exists a unique function $\varphi\in\mathcal{E}(M,\omega_0)$, such that $\mu=(\omega_0+\sqrt{-1}\partial
\bar{\partial}\varphi)^n$. 
\end{thm}
We refer the readers to \cite{GZ} for more details about the definition of the space $\mathcal{E}$. Among other properties, the authors also showed that any $\varphi\in\mathcal{E}$ has zero Lelong number everywhere. Hence Proposition \ref{p4.1} follows from the Skoda's integrability theorem.

Here we will give a more elementary and direct proof of Proposition \ref{p4.1} which avoids the use of pluripotential machinery. Moreover, it has the merit of explicit estimates.
The idea is similar to Theorem 5.2 of \cite{Chen-Cheng}.

To start, denote $A=\int_Me^F\Phi(e^F)\omega_0^n$ and let $\psi\in psh(\omega_0)$ be the solution of the equation:
\begin{equation*}
\begin{split}
&(\omega_0+\sqrt{-1}\partial\bar{\partial}
\psi)^n=\frac{e^F\Phi(e^F)}{A}\omega_0^n,\\
&\sup_M\psi=0.
\end{split}
\end{equation*}
Proposition \ref{p4.1} will follow from the following lemma:
\begin{lem}\label{l4.2}
For any $0<\eps<\frac{1}{2}$, there exists a constant $C_{\eps}$, which has the same dependence as described in Proposition \ref{p4.1} but additionally on $\eps$, such that $\varphi-\eps\psi\ge -C_{\eps}$.
\end{lem}
To see that Proposition \ref{p4.1} follows from this lemma, we recall the following lemma of Tian, which is based on the $L^2$ estimate of H\"ormander:
\begin{lem}\label{l4.3}
Let $(M,\omega_0)$ be a compact K\"ahler manifold. Then there exists a constant $\alpha_0>0$, $C_0>0$, such that for all $\varphi\in psh(\omega_0)$, we have $\int_Me^{-\alpha_0(\varphi-\sup_M\varphi)}\omega_0^n\le C_0$.
\end{lem}
For any $p>0$, we see from Lemma \ref{l4.2} that
\begin{equation*}
-p\varphi\le pC_{\eps}-p\eps\psi.
\end{equation*}
Hence if we take $\eps$ such that $p\eps=\alpha_0$ and apply Lemma \ref{l4.3}, we see that
\begin{equation*}
\int_Me^{-p\varphi}\omega_0^n\le e^{pC_{\eps}}\int_Me^{-p\eps\psi}\omega_0^n\le e^{pC_{\eps}}C_0.
\end{equation*}
Now it only remains to prove Lemma \ref{l4.2}.
\begin{proof}
First we can compute
\begin{equation*}
\begin{split}
\Delta_{\varphi}&\psi=g_{\varphi}^{i\bar{j}}
\partial_{i\bar{j}}\psi=g_{\varphi}^{i\bar{j}}(g_{i\bar{j}}+\psi_{i\bar{j}})-tr_{\varphi}g\ge n\big(\det g_{\varphi}^{i\bar{j}}\det(g_{i\bar{j}}+\psi_{i\bar{j}})\big)^{\frac{1}{n}}-tr_{\varphi}g\\
&\ge n\big(e^{-F}e^F\Phi(e^F)A^{-1}\big)^{\frac{1}{n}}-tr_{\varphi}g=n\Phi^{\frac{1}{n}}(e^F)A^{-\frac{1}{n}}-tr_{\varphi}g.
\end{split}
\end{equation*}
Therefore
\begin{equation*}
\Delta_{\varphi}(\varphi-\eps\psi)=n-tr_{\varphi}g-\eps\Delta_{\varphi}\psi\le n-\eps n\Phi^{\frac{1}{n}}(e^F)A^{-\frac{1}{n}}-(1-\eps)\eps tr_{\varphi}g.
\end{equation*}
Suppose that the function $\varphi-\eps\psi$ achieves minimum at $p_0$. We may choose a ball $B_{d_0}(p_0)$ centered at $p_0$, such that under normal coordinates at $p_0$, we have $\frac{1}{2}\le \det g_{i\bar{j}}\le 2$. Let $\eta:M\rightarrow\mathbb{R}^+$ be a cut-off function on $M$, such that $\eta(p_0)=1$, $\eta\equiv 1-\theta$ outside the ball $B_{d_0}(p_0)$, with $|\nabla\eta|\le\frac{\theta}{d_0}$, $|D^2\eta|\le\frac{\theta}{d_0^2}$. Here $\theta>0$ is a parameter to be chosen(small).

Now we choose $\delta=\frac{\alpha_0}{2n}$, where $\alpha_0>0$ is the $\alpha$-invariant given by Lemma \ref{l4.3}. We can compute
\begin{equation*}
\begin{split}
&\Delta_{\varphi}\big(
e^{-\delta(\varphi-\eps\psi)}\big)=
e^{-\delta(\varphi-\eps\psi)}\delta^2|\nabla_{\varphi}(\varphi-\eps\psi)|^2_{\varphi}-\delta e^{-\delta(\varphi-\eps\psi)}\Delta_{\varphi}(\varphi-\eps\psi)\\
&\ge \delta^2 e^{-\delta(\varphi-\eps\psi)}|\nabla_{\varphi}(\varphi-\eps\psi)|^2_{\varphi}+
e^{-\delta(\varphi-\eps\psi)}\delta(\eps n\Phi^{\frac{1}{n}}(e^F)A^{-\frac{1}{n}}+(1-\eps)tr_{\varphi}g-n).
\end{split}
\end{equation*}
Hence
\begin{equation}\label{4.1}
\begin{split}
\Delta_{\varphi}&\big(
e^{-\delta(\varphi-\eps\psi)}\eta\big)
=e^{-\delta(\varphi-\eps\psi)}\delta^2|\nabla_{\varphi}(\varphi-\eps\psi)|^2_{\varphi}\eta\\
&+e^{-\delta(\varphi-\eps\psi)}
\delta\big(\eps n\Phi^{\frac{1}{n}}(e^F)A^{-\frac{1}{n}}+(1-\eps)tr_{\varphi}g-n\big)\eta+
e^{-\delta(\varphi-\eps\psi)}
\Delta_{\varphi}\eta\\
&+e^{-\delta(\varphi-\eps\psi)}(-\delta)\nabla_{\varphi}(\varphi-\eps\psi)
\cdot_{\varphi}\nabla_{\varphi}\eta\\
&\ge e^{-\delta(\varphi-\eps\psi)}\delta\big(\eps n\Phi^{\frac{1}{n}}(e^F)A^{-\frac{1}{n}}+(1-\eps)tr_{\varphi}g-\frac{|D^2\eta|}{\eta}tr_{\varphi}g-\frac{|\nabla\eta|^2}{\eta^2}tr_{\varphi}g-n\big).
\end{split}
\end{equation}
We then choose the parameter $\theta$ sufficiently small so that
\begin{equation*}
(1-\eps)-\frac{\theta}{d_0^2(1-\theta)}-\frac{\theta^2}{d_0^2(1-\theta)^2}>0.
\end{equation*}
With this choice, we conclude from (\ref{4.1}) that
\begin{equation}
\Delta_{\varphi}\big(e^{-\delta(
\varphi-\eps\psi)}\eta\big)\ge e^{-\delta(\varphi-\eps\psi)}\delta \big(\eps n\Phi^{\frac{1}{n}}(F)A^{-\frac{1}{n}}-n\big).
\end{equation}
Apply the Alexandrov estimate in the ball $B_{d_0}(p_0)$, we obtain that
\begin{equation}\label{4.3}
\begin{split}
\sup_{B_{d_0}(p_0)}&e^{-\delta(\varphi-\eps\psi)}\eta\le \sup_{\partial B_{d_0}(p_0)}e^{-\delta(\varphi-\eps\psi)}\eta\\
&+C_n
d_0\bigg(\int_{B_{d_0}(p_0)}e^{-2n\delta(\varphi-\eps\psi)}\delta^{2n}
\big((\eps n\Phi^{\frac{1}{n}}(F)A^{-\frac{1}{n}}-n)^-\big)^{2n}e^{2F}
\omega_0^n\bigg)^{\frac{1}{2n}}.
\end{split}
\end{equation}
In order for $(\eps n\Phi^{\frac{1}{n}}(F)A^{-\frac{1}{n}}-n)^-$ to be nonzero, we must have that
\begin{equation*}
\eps n\Phi^{\frac{1}{n}}(F)A^{-\frac{1}{n}}-n<0.
\end{equation*}
This gives an upper bound for $F$, say $F\le C_1$, where $C_1$ has the said dependence as in the lemma.
Hence
\begin{equation}
\begin{split}
\textrm{The integral in (\ref{4.3})}&\le \int_{B_{d_0}(p_0)\cap\{F\le C_1\}}e^{-2n\delta(\varphi-\eps\psi)}
\delta^{2n}n^{2n}e^{2C_1}\omega_0^n\\
&\le \delta^{2n}n^{2n}e^{2C_1}
\int_Me^{-\alpha_0\varphi}\omega_0^n\le \delta^{2n}n^{2n}e^{2C_1}C_0.
\end{split}
\end{equation}
On the other hand, note that $\eta\le 1-\theta$ on $\partial B_{d_0}(p_0)$, and $e^{-\delta(\varphi-\eps\psi)}$ achieves maximum over $M$ at $p_0$. Then (\ref{4.3}) implies:
\begin{equation}
\sup_Me^{-\delta(\varphi-\eps\psi)}\le \frac{1}{\theta}C_nd_0\delta^{2n}n^{2n}e^{2C_1}C_0.
\end{equation}
This finishes the proof.
\end{proof}

\end{document}